\documentclass[11pt,english]{article}
\usepackage{amsmath}
\usepackage{amssymb}
\usepackage{amsthm}
\usepackage[english]{babel}
\usepackage{times}
\usepackage[T1]{fontenc}
\usepackage[all]{xy}
\usepackage{geometry}
\theoremstyle{plain}
\newtheorem{teo}{Theorem}[section]

\newtheorem{pro}[teo]{Proposition}

\newtheorem{defn}[teo]{Definition}

\theoremstyle{definition}

\newcommand{\B}{\mathbb{B}}
\newcommand{\HH}{\mathbb{H}}
\newcommand{\N}{\mathbb{N}}
\newcommand{\C}{\mathbb{C}}
\newcommand{\rr}{\mathbb{R}}
\newcommand{\s}{\mathbb{S}}
\newcommand{\Z}{\mathbb{Z}}

\newcommand{\p}{\partial}

\DeclareMathOperator{\ess}{ess}

\DeclareMathOperator{\ext}{ext}

\newcommand{\RR}{\mathbb{R}}
\newcommand{\BB}{\mathbb{B}}

\renewcommand{\SS}{\mathbb{S}}

\title{\bf Quaternionic Hankel operators and approximation by slice regular functions}
\author{ Giulia Sarfatti \thanks{Partially supported by INDAM-GNSAGA, by the PRIN project Real and Complex Manifolds, by the FIRB project Differential Geometry and Geometric Function Theory, and by the SIR
2014 project Analytic Aspects of Complex and Hypercomplex Geometry (code
no. RBSI14DYEB) of the Italian MIUR. }\\
\normalsize Dipartimento di Matematica e Informatica ``U. Dini'',
Universit\`a di Firenze\\
\normalsize viale Morgagni 67/a
50134 Firenze, Italy,  \ email:  sarfatti@math.unifi.it }

\date{}

\begin{document}

\maketitle

\begin{abstract}
In this paper we study Hankel operators in the quaternionic setting. In particular we prove that they  can be exploited to measure the $L^{\infty}$ distance of a slice $L^{\infty}$ function (i.e. an essentially bounded function on the quaternionic sphere $\p\B$ which is affine with respect to quaternionic imaginary units) from the space of bounded slice regular functions (i.e. bounded quaternionic power series on the quaternionic unit ball $\B$). Among the difficulties arising from the non-commutative context there is the lack of a good factorization result for slice regular functions in the Hardy space $H^1$. 
\end{abstract}

\section{Introduction}
A linear operator on a Hilbert space is called a Hankel operator if its associated matrix 
has constant antidiagonals. 
Given any sequence $\alpha=\{\alpha_n\}_{n\in \N}$ with values in the skew field of quaternions,  we can form the infinite matrix $M_\alpha=(\alpha_{j+k})_{j,k=0}^{+\infty}$ with constant antidiagonals, and let $\Gamma_\alpha$ acting on quaternion valued sequences $v=(v_j)_{j=0}^{+\infty}$ by matrix multiplication: $(\Gamma_\alpha v)(j)=\sum_{k=0}^{\infty}\alpha_{j+k}v_k$. 
The classical Nehari Theorem (see \cite{nehari}) characterizes the (complex valued) sequences $\alpha$ such that the Hankel operator $\Gamma_\alpha$ is bounded on the Hilbert space $\ell^2(\N,\C)$. 
Together with Nicola Arcozzi, we investigated this problem in the quaternionic setting in \cite{carleson}. We showed that, in analogy with the complex valued case, where the problem can be translated in terms of holomorphic functions,  in the quaternionic setting it can be reformulated in terms of slice regular functions.
This class of function, introduced by Gentili and Struppa in \cite{GSAdvances}, represents in fact a valid counterpart in the quaternionic setting to holomorphic functions. We refer to the monograph \cite{libroGSS} for all results and proofs concerning the theory of slice regular functions. For a survey of the theory of Hankel operators in the complex setting see \cite{peller}.


The purpose of the present work is to give an interpretation of quaternionic Hankel operators as tools to measure the distance of a bounded quaternionic {\em slice} function from the space of bounded {\em slice regular} functions.

Let $\HH$ denote the skew field of quaternions, let $\BB=\{q\in\HH:\ |q|<1\}$ be the quaternionic unit ball and
let $\partial \BB$ be its boundary, containing elements of the form $q=e^{tI},\ I\in\SS,\ t\in\RR$, where $\SS=\{q\in \HH : q^2=-1\}$ is the two dimensional sphere of imaginary units in $\HH$.
We endow $\partial\BB$ with the measure $d\Sigma\left(e^{tI}\right)=d\sigma(I)dt$, which is naturally associated with the Hardy space $H^2(\BB)$ of slice regular functions on $\BB$, see \cite{metrica}. 
The measures are normalized so that $\Sigma(\partial\BB)=\sigma(\SS)=1$. 

The class of functions we are considering satisfies the following algebraic condition:
a function $f$ defined on $\p\B$ is called a {\em slice function} if for any two sphere of the type $e^{t\s }:=\{e^{tI} : I\in\s\}$ contained in $\p\B$, 
\[f(e^{tJ})=a(t)+Jb(t),\]
where $a,b$ are quaternion valued functions depending only on $t$. 
Namely $f$ is slice if its restriction to each sphere $e^{t\s}$ is affine in the imaginary unit variable. 
Slice functions were introduced and studied by Ghiloni and Perotti in the more general setting of real alternative algebras in \cite{ghiloniperotti}.   
The restriction to $\p\B$ of slice regular functions furnishes an important class of examples of slice function.
A concrete example is given by convergent (for almost every $t$) power series of the form $e^{It}\mapsto \sum_{n\in \mathbb{Z}}e^{Int}a_n$, with quaternionic coefficients $a_n$. For instance, such a series does converge if the sequence of coefficients $\{a_n\}$ belongs to $\ell^2(\Z, \HH)$.  
Power series with coefficients in $\ell^{1}(\Z,\HH)$ are considered in \cite{wiener}. 

Let $L_s(\p\B)$ denote the space of measurable slice functions and, for $1\le p<+\infty$, let
$L^p(\partial\BB)$ denote the space of (equivalence classes of) functions $f:\partial\BB\to\HH$ such that $\|f\|_p^p:=\int_{\partial\BB}|f|^pd\Sigma<\infty$.  
For $p=+\infty$, $L^{\infty}(\p\B)$ denotes the space of (equivalence classes of) measurable functions essentially bounded with respect to the measure $d\Sigma$.
The space of slice $L^p$ functions will be denoted by $L^p_s(\p\B)=L^p(\p\B)\cap L_s(\p\B)$. It is possible to show that $L^p_s(\p\B)$ has a natural structure of right linear space over $\HH$ (to be more precise of a right $\HH$ modulus), and of a Banach space.

In the case $p=2$, we have
\begin{pro}
\[L^2_s(\p\B)=\Big\{f\in L^2(\p\B) : \ \text{for a.e. $q\in \p\B$,}\ f(q)=\sum_{n\in \Z}q^na_n  \ \text{with }\ \{a_n\}\in \ell^2(\Z, \HH) \Big\}\]
and the inner product
\[\Big\langle \sum_{n\in \Z}q^na_n, \sum_{n\in \Z}q^nb_n\Big\rangle_{L^2_s(\p\B)}:=\sum_{n\in \Z}\overline{b}_na_n\]
endows $L^2_s(\p\B)$ with the structure of a quaternionic Hilbert space.
\end{pro}
For the extention of classical functional analysis results to the quaternionic setting we refer to the book \cite{Brackx} and, for the specific case of slice functions, to \cite{ghilonimorettiperotti}.

The first result concerning the study of Hankel operators is a reformulation of the Nehari type Theorem proved in \cite{carleson} in which we give a characterization of bounded Hankel operators is given in terms of slice $L^{\infty}$ functions. 
The norm of a bounded linear operator $T$ mapping a quaternionic normed linear space $\mathcal U$ to another quaternionic normed linear space $\mathcal V$ is defined as usual as 
\[\|T\|_{\mathcal{B}(\mathcal U)}:=\sup_{u\in \mathcal U, u\not \equiv 0}\frac{\|Tu\|_{\mathcal V}}{\|u\|_{\mathcal U}}.\] 
\begin{teo}\label{nehariintro}
Let $\alpha=\{\alpha_n\}_{n\in \N}\subset \HH$ be a quaternion valued sequence. The operator $\Gamma_{\alpha }$ mapping any sequence $a=\{a_n\}\in \ell^{2}(\N, \HH)$ to the sequence defined by
\[ (\Gamma_\alpha a)(n):=\sum_{j\ge 0}\alpha_{n+j}a_j\]
 is bounded on $\ell^2(\N, \HH)$ if and only if there exists $\psi(q)=\sum_{n\in \Z}q^n\widehat{\psi}(n) \in L_{s}^{\infty}(\p\B)$ such that $\widehat{\psi}(m)=\overline{\alpha}_m$ for any $m\ge 0$.
In this case 
\begin{equation} \label{normaGammaintro}
\begin{aligned}
\inf\{\|\psi\|_{L^{\infty}(\p\B)} \,& : \, \psi \in L^{\infty}_s(\p\B), \widehat{\psi}(m)= \overline{\alpha}_m,  m\ge 0\}\le \|\Gamma_{\alpha}\|_{\mathcal{B}(\ell^2(\N, \HH))}\\
&\le 2 \inf\{\|\psi\|_{L^{\infty}(\p\B)} \, : \, \psi \in L^{\infty}_s(\p\B),  \widehat{\psi}(m)= \overline{\alpha}_m,  m\ge 0\}.
\end{aligned}
\end{equation}
\end{teo}
\noindent The proof of the ``if and only if'' part of the statement exploits two results. The first one is Theorem 1.1 in \cite{carleson}, which gives several equivalent conditions on the sequence $\alpha$ and its generating function $q \mapsto \sum_{n\in \N}q^n\overline{\alpha}_n$ in order to have a bounded operator $\Gamma_\alpha$. 
In particular we will use the following:
\begin{teo}\label{passouno}
Let $\alpha=\{\alpha_n\}_{n\in \N}\subset \HH$ be a quaternion valued sequence. Then the operator $\Gamma_{\alpha }$ is bounded on $\ell^2(\N, \HH)$ if and only if the function $q \mapsto \sum_{n\in \N}q^n\overline{\alpha}_n$ belongs to the space $BMOA(\p\B)$.
\end{teo}
\noindent The symbol $BMOA(\p\B)$ denotes the space of slice regular functions on $\B$ with bounded mean oscillation. See Section \ref{LP} for the precise definition.
The second result exploited is a representation result for quaternionic $BMO$ functions proved in Section \ref{LP}. 
The main point here relies in the estimates \eqref{normaGammaintro} of the norm of $\Gamma_\alpha$. 
In the proof of the second inequality in \eqref{normaGammaintro} the lack of a good factorization result for slice regular functions in the Hardy space $H^1(\B)$ requires some attention. Moreover a duality result that may have an independent interest is needed:
\begin{teo}\label{dualsintro}
The dual space $(L^1_{s}(\p\B))^*$ is isometrically isomorphic to $L^{\infty}_{s}(\p\B)$. 
\end{teo}

Hankel operators admit also a realization as operators from the quaternionic Hardy space (viewed as a subspace of $L^2_s(\p\B)$)  
\[H^2(\p\B)=\Big\{ f \in L^2(\p\B) :\ \text{for a.e. $q\in\p\B,$}\  f(q)=\sum_{n\ge 0}q^na_n \ \text{with}\ \{a_n\}\in \ell^2(\N, \HH) \Big\}\] to its orthogonal complement $H^2_-(\p\B)$ in $L^2_s(\p\B)$.  
%
%
Given any $\varphi(q)=\sum_{n\in\Z}q^n\widehat{\varphi}(n)\in  L^2_{s} (\p\B) $ we can define the operator  
$H_{\varphi}: H^2(\p\B) \to H^2_- (\p\B)$
as
\begin{equation}\label{projectionin}
H_{\varphi}f:=\mathbb{P}_- ( \varphi \star f )=\sum_{n < 0}q^n\sum_{k \ge 0}\widehat{\varphi}(n-k)\widehat{f}(k).
\end{equation} 
where the symbol $\star$ denotes an appropriate multiplication operation between slice functions, and $\mathbb{P}_-$ is the orthogonal projection from $L^2_s(\p\B)$ onto $H^2_-(\p\B)$.

Translating Theorem \ref{nehariintro} from sequences in $\ell^2(\N, \HH)$  to functions in $L^2_s(\p\B)$, allows us to prove our main result.
\begin{teo}\label{mainintro}
Let $\varphi \in  L^{\infty}_{s}(\p\B)$. Then 
\begin{equation}\label{infintro}
\|H_\varphi\|_{\mathcal B (H^2(\p\B))}=\inf\{\|\varphi-f\|_{L^{\infty}(\p\B)} \, : \, f \in H^{\infty}(\p\B)
\}.
\end{equation}
\end{teo}
\noindent Here $H^{\infty}(\p\B)$ denotes the space of bounded slice regular functions on $\B$, introduced in \cite{hardy}, viewed as a subspace of $L^{\infty}_s(\p\B)$. 
By compactness, the infimum in equation \eqref{infintro} is attained for any $\varphi \in L^{\infty}_{s}(\p\B)$: there exists always a regular function $g\in H^{\infty}(\p\B)$ such that 
\[\|H_\varphi\|_{\mathcal B (H^2(\p\B))}=\|\varphi-g\|_{L^{\infty}(\p\B)}={\rm dist}_{L^{\infty}}(\varphi, H^{\infty}(\p\B)) .\]
A function $g$ realizing the distance of $\varphi$ from $H^{\infty}(\p\B)$ is called an {\em $L^{\infty}$-best approximation of $\varphi$ 
by a slice regular function}.

We conclude this article with some further results. The first one concerns the uniqueness of a best approximation, and the second one gives a characterization of bounded Hankel operators in terms of shift operators.

%

The paper is structured as follows: in Section \ref{prelim} we give some background on slice and slice regular functions; in Section \ref{LP} we establish the setting of our study, we introduce slice $L^p$ and $BMO$ (equivalence classes of) measurable functions and we prove some functional analytical results needed in the sequel; Section \ref{main} is devoted to prove Theorems \ref{nehariintro} and \ref{mainintro}, together with  
some further results concerning bounded Hankel operators.
 
\section{Preliminaries}\label{prelim}
In this section we recall the definitions of slice and of slice regular functions over the quaternions $\HH$, together 
with some basic properties. 
Let $\s$ denote the two-dimensional sphere of imaginary units of $\HH$, $\s=\{q\in \HH \, | \, q^2=-1\}$. 
One can ``slice'' the space $\HH$ in copies of the complex plane that intersect along the real axis,  
\[ \HH=\bigcup_{I\in \s}(\rr+\rr I),  \hskip 1 cm \rr=\bigcap_{I\in \s}(\rr+\rr I),\]
where $L_I:=\rr+\rr I\cong \C$, for any $I\in\s$.  
Each element $q\in \HH$ can be expressed as $q=x+yI_q$, where $x,y$ are real (if $q\in\rr$, then $y=0$) and $I_q$ is an imaginary unit. To have uniqueness outside the real axis we can choose $y>0$. The conjugate of $q$ is $\bar q =x-yI$ and its modulus is given by $|q|^2=q\bar q=x^2+y^2$. Every non-zero quaternion has a multiplicative inverse, denoted by $q^{-1}$, that can be computed as $q^{-1}=\overline{q}/|q|^2$, hence providing $\HH$ with the structure of a skew field over $\rr$.
For any $\Omega\subseteq \HH$ and  for every $I\in\s$ we will denote by $\Omega_I$ the intersection $\Omega \cap L_I$.

Let us focus our attention to quaternion valued functions defined on the boundary $\p\B$ of the quaternionic unit ball $\B:=\{q\in \HH \ : \ |q|<1\}$. This domain present an important symmetry property: for any quaternion $x+yI$ contained in $\p\B$, the entire two dimensional sphere $x+y\s$ is contained in $\p\B$. On subsets presenting such a symmetry it is possible to give the definition of slice functions.    
This class of functions was introduced and studied in a more general setting by Ghiloni and Perotti in \cite{ghiloniperotti}. We do not intend to give here all the details. 
In this paper a function $f:\p\B \to \HH$ will be called a {\em slice function} if it is affine on spheres of the form $x+y\s$ contained in $\p\B$ with respect to the imaginary unit. More precisely, $f$ is slice if for any $I,J\in \s$ and for any $x+yI\in \p\B$
\begin{equation}\label{repfor}
\begin{aligned}
f(x+yJ)&=\frac{1}{2}\left(f(x+yI)+f(x-yI)\right)+\frac{J I}{2}\left(f(x-yI)-f(x+yI)\right)\\
&=\frac{1-JI}{2}f(x+yI)+\frac{1+JI}{2}f(x-yI).
\end{aligned}
\end{equation}
In the sequel we will adopt equation \eqref{repfor} as definition of slice functions.

A concrete example of slice functions is given by convergent (for almost every $t$) power series of the form $e^{It}\mapsto \sum_{n\in \mathbb{Z}}e^{Int}a_n$, with quaternionic coefficients $a_n$ ( in fact, for any $n$, $e^{Int}=\cos(nt)+ I\sin(nt)$).

The restriction to $\p\B$ of slice regular functions on open balls centered at the origin $B_r=\{q\in \HH \ : \ |q|<r\}$ gives an important class of examples of slice function:
a function $f:B_r \to \HH$ is called {\em slice regular} if for any $I\in\s$ the restriction $f_I$ of $f$ to $B_r\cap L_I$
is {\em holomorphic}, i.e. it has continuous partial derivatives and it is such that
\[\overline{\p}_If_I(x+yI)=\frac{1}{2}\left(\frac{\p}{\p x}+I\frac{\p}{\p y}\right)f_I(x+yI)=0\]
for all $x+yI\in B_r
\cap L_I$.
In fact (see \cite{libroGSS}), a function $f$ is slice regular on $B_r$ if and only if it has a power series expansion
$f(q)=\sum^{\infty}_{n= 0}q^na_n$ converging in $B_r$ and such functions clearly satisfy representation formula \eqref{repfor} on the entire $B_r$.

We point out that equation \eqref{repfor} also furnishes a tool to uniquely extend a function $f_I$ defined on a circle $\p\B_I$ to a slice function $f:=\ext(f_I)$ defined on the entire $\p\B$,
\begin{equation*} 
\ext(f_I)(x+yJ)=\frac{1-JI}{2}f_I(x+yI)+\frac{1+JI}{2}f_I(x-yI).
\end{equation*}
The operator $\ext$ was introduced in the setting of slice regular functions in \cite{ext}.
Since pointwise product of functions does not preserve the classes of slice and of slice regular functions, 
a new multiplication operation is defined, the so called {\em regular} or $\star$-{\em product}. In the case of slice regular functions defined on $B_r$ it has an expression given in terms of their power series expansion,  which can be naturally extended to (a.e.) convergent power series of the form 
$\sum_{n\in \mathbb{Z}}qa_n$, with $|q|=1$. 
Let $f(q)=\sum_{n\in \Z}q^n a_n$ and $ g(q)=\sum_{n\in \Z}q^n b_n$ be two slice functions on $\p\B$. Then
\[f \star g (q):= \sum_{n\in \Z}q^n \sum_{k \in \Z} a_k b_{n-k}.\]
%
As for slice regular power series the $\star$-product is related to the standard pointwise product by the following formula.
\begin{pro}\label{trasf}
Let $f,g$ be slice functions on $\p\B$ admitting expansions $f(q)=\sum_{n\in \Z}q^n a_n$  and $g(q)=\sum_{n\in \Z}q^n b_n$. 
Then
\[f\star g(q)=\left\{\begin{array}{l r}
0 & \text{if $f(q)=0 $}\\
f(q)g(f(q)^{-1}qf(q)) & \text{if $f(q)\neq 0 $}
                \end{array}\right.
 \] 
\end{pro}
\noindent The proof follows by the same computation that gives the result in the slice regular case (see\cite{libroGSS}).
Also the notions of regular conjugation and of symmetrization can be generalized to the case of power series with negative powers. If $f(q)=\sum_{n\in \Z}q^n a_n$ is a slice function on $\p\B$, the {\em regular conjugate} and the {\em symmetrization} of $f$ are 
\[f^c(q):=\sum_{n\in \Z}q^n \overline{a}_n \ \text{ and }\  f^s(q):=f^c\star f (q)=f\star f^c (q) \ \text{ respectively.}\] 
The reciprocal $f^{-\star}$ of $f(q)=\sum_{n\in \Z}q^na_n$ with respect to the $\star$-product is then given, as in the regular case, by
\[f^{-\star}(q)=\frac{1}{f\star f^c(q)}f^c(q).\]
The function $f^{-\star}$ is defined on $\p\B \setminus \{q\in \p\B \ | \ f\star f^c(q)=0\}$ and $f\star f^{-\star}=f\star f^{-\star}=1$. 
We conclude this preliminary section with the following proposition, that can be proved by direct computation, as in the case of slice regular power series.
\begin{pro}\label{scambio}
Let $f,g$ be slice functions on $\p\B$ admitting expansions $f(q)=\sum_{n\in \Z}q^n a_n$  and $g(q)=\sum_{n\in \Z}q^n b_n$. 
Then $(f\star g)^c=g^c\star f^c.$
\end{pro}

\section{Slice $L^p$ functions}\label{LP}
In this section we establish the setting of our study. In particular we introduce the spaces of equivalence classes of {\em slice} $L^p$ functions and that of {\em slice BMO} functions.  

Let $I\in \s$. For $1\le p< +\infty$, we will denote by $L^p(\p\B_I)$ the space of quaternion valued (equivalence classes of) measurable 
functions $f:\p\B_I \to \HH$ such that $\|f\|^p_{L^p(\p\B_I)}:=\int_{0}^{2\pi}|f(e^{It})|^pdt<+\infty$ and by $L^p(\p\B)$ the space of (equivalence classes of) measurable functions $f:\p\B \to \HH$ such that \[\|f\|^p_{L^p(\p\B)}:=\int_{\p\B}|f(q)|^pd\Sigma(q)<+\infty,\] 
where $d\Sigma$ is the volume form naturally associated with the Hardy space $H^2(\B)$, 
introduced in \cite{metrica}: if $q=e^{It}$, then $d\Sigma(q)=d\sigma(I)dt$ and $d\Sigma$, $d\sigma$ are normalized so that $d\Sigma(\p\B)=d\sigma(\s)=1$. 
For $p=\infty$, we will denote by $L^{\infty}(\p\B_I)$ the space of essentially bounded (equivalence classes of) measurable functions on $\p\B_I$. 
The space $L^{\infty}(\p\B)$ contains (equivalence classes of) measurable functions $f:\p\B\to \HH$ such that
\[||f||_{L^{\infty}(\p\B)}:=\ess \!\sup_{q\in \p\B}|f(q)|<+\infty,\]
where the essential supremum is taken with respect to the measure $d\Sigma$:
\[\ess \!\sup_{q\in \p\B}|f(q)|=\inf\{ \lambda \ge 0 \, : \, \Sigma\left(\{q\in \p\B \, : \, |f(q)|>\lambda\}\right)=0 \}.\]

Quaternionic $L^p$ spaces share many of the basic properties of classical complex $L^p$ spaces. 
In this paper we will not investigate the general theory of these function spaces, in particular we will focus our attention to their subspaces of slice functions. We recall here that in \cite{proj} the $L^p$ norm of the orthogonal projection from the space $L^2(\p\B)$ to its closed subspace $L^2_s(\p\B)$ is studied.

A first (easy) remark is that the restriction to a single slice of an element of $L^p(\p\B)$ does not necessarily belong to $L^p(\p\B_I)$ for every $I\in \s$ (but only for $\sigma$-almost every $I$), while clearly the opposite implication holds true.
The two notions coincide in the case of slice functions.
For $1\le p \le +\infty$, let $L_s^p(\p\B)$ denote the space of measurable slice functions bounded in $L^p$-norm.
Using representation formula \eqref{repfor} it is not difficult to prove
\begin{pro}\label{sliceLP}
For any $1\le p\le +\infty$, a slice function $f:\p\B\to \HH$ belongs to $L^p(\p\B)$ if and only if it belongs to $L^p(\p\B_I)$ for one and hence for any $I\in \s$.
\end{pro}
As in the case of $H^p$ spaces of slice regular functions, studied in \cite{hardy}, it can be proven that $L_s^p(\p\B)$ spaces have a natural structure of quaternionic right Banach spaces. 
%
In the case $p=2$, as in the case of the Hardy space $H^2(\B)$ (see \cite{milanesishur} and \cite{milanesipontryagin}) the inner product 
\[\langle f,g \rangle_{L^2(\p\B)}:=\int_{\p\B}\overline{g}f d\Sigma\]
gives to $L^2(\p\B)$ the structure of a quaternionic Hilber space.
Moreover, when considering slice $L^2$ functions, we obtain the following result. 
\begin{pro}
\[ L^2_{s} (\p \B)=\Big\{f\in L^2(\p\B) : \ \text{for a.e. $q\in \p\B$,}\ f(q)=\sum_{n\in \Z}q^na_n  \ \text{with }\ \{a_n\}\in \ell^2(\Z, \HH) \Big\}.\]
In particular, $\{q^n\}_{n\in \Z}$ is an orthonormal basis of $ L^2_{s} (\p\B)$.
\end{pro}
\begin{proof}
On the one hand, il $f \in L^2(\p\B)$, recalling that $d\Sigma$ is the product of $dt$ with $d\sigma$, we get that there exists $I\in \s$ such that $f\in L^2(\p\B_I)$. Then if $f$ splits on $L_I$ with respect to $J\in \s$, $J\perp I$, as 
$f(z)=F(z)+G(z)J$, the orthogonality of $I$ and $J$ guarantees that the splitting components $F,G:\p\B_I\to L_I$ are complex $L^2$ functions. Therefore they admit a representation of the form 
\[F(z)=\sum_{n\in \Z}z^n\widehat F(n), \quad G(z)=\sum_{n\in \Z}z^n\widehat G(n) \]  
where $z\in \p\B_I$ and $\widehat F (n)$, $\widehat G (n)$ denote the $n$-th Fourier coefficient of $F$ and $G$ respectively.
For any $z\in \p\B_I$ we can therefore write
\[f(z)=\sum_{n\in \Z}z^n(\widehat F (n)+\widehat G (n)J)=:\sum_{n\in \Z}z^n\widehat f (n),\]
where the sequence $\{\widehat f (n)\}_{n\in \Z}$ belongs to $\ell^2(\Z,\HH)$ thanks to fact that $J \perp I$ and $\{\widehat F (n)\}_{n\in \Z}$ and $\{\widehat G (n)\}_{n\in \Z}$ belong to $\ell^2(\Z,\C)$.
If furthermore $f$ satisfies equation \eqref{repfor}, since every $q=x+yJ\in \p\B$ can be written as $q=\frac{1-JI}{2}(x+yI)+\frac{1+JI}{2}(x-yI)$, then the power series expression of $f$ can be extended to the whole $\p\B$ as
\[f(q)=\sum_{n\in \Z}q^n\widehat f(n).\]
Notice that this also shows that the coefficients $\widehat{f}(n)$ do not depend on the choice of the slice $L_I$.

On the other hand, if $f(q)=\sum_{n\in \Z}q^na_n $ with $q\in \p\B$ and $\{a_n\}_{n\in \Z}\in \ell^2(\Z,\HH)$, then the restriction $f_I$ of $f$ to $L_I$ belongs to $L^2(\p\B_I,\HH)$ for any $I\in \s$. In fact, for any $n,m\in \Z$, $\int_{0}^{2\pi}e^{-ntI}e^{mtI}dt=2\pi \delta_{n}^{m}$ (which proves the orthonormality of functions $q\mapsto q^n$), and hence
\[\Big\|\sum_{n\in \Z}q^n a_n \Big\|^2_{L^2(\p\B_I, \HH)}=\sum_{n\in \Z}|a_n|^2=\left\|\{a_n\}_{n\in \Z}\right\|^2_{\ell^2(\Z,\HH)}<+\infty.\]
Equality $x+yJ=\frac{1-JI}{2}(x+yI)+\frac{1+JI}{2}(x-yI)$, for any $I,J\in \s$, immediately implies that $f$ is a slice function.
\end{proof}

In particular $ L^2_{s} (\p\B)$ endowed with the inner product  
\[ \Big \langle \sum_{n\in \Z}q^n a_n,\sum_{n\in \Z}q^n b_n\Big \rangle_{ L^2_{s} (\p\B)}=\sum_{n\in \Z}\bar b_n a_n\]
is a quaternionic Hilbert space.
We point out that for any $f(q)=\sum_{n\in \Z}q^n a_n, g(q)=\sum_{n\in \Z}q^n b_n \in  L^2_{s} (\p\B)$, their inner product can also be expressed in an integral form as
\[\Big \langle \sum_{n\in \Z}q^n a_n,\sum_{n\in \Z}q^n b_n\Big \rangle_{ L^2_{s} (\p\B)}=\frac{1}{2\pi}\int_{0}^{2\pi}\overline{g(e^{I\theta})}f(e^{I\theta})d\theta \]
where $I$ is any imaginary unit.   
Since it does not depend on the imaginary unit, we can also write 
\begin{align}\label{normaint}
\Big \langle \sum_{n\in \Z}q^n a_n,\sum_{n\in \Z}q^n b_n\Big \rangle_{ L^2_{s} (\p\B)}&=\int_{\s}d\sigma\frac{1}{2\pi}\int_{0}^{2\pi}\overline{g(e^{I\theta})}f(e^{I\theta})d\theta=\int_{\p\B}\overline{g(q)}f(q)d\Sigma(q)= \langle f, g \rangle_{ L^2(\p\B)}. 
\end{align}
Notice that the $L^2$ norm of a function $f\in L^2_s(\p\B)$ depends only on the moduli of the coefficients of its power series expansion. Hence $f\in L^2(\p\B)$ if and only if its regular conjugate $f^c$ does, and moreover 
\[\|f\|^2_{L^2_s(\p\B)}=\sum_{n\in \Z}|\widehat{f}(n)|^2=\sum_{n\in \Z}\left|\overline{\widehat{f}(n)}\right|^2=\|f^c\|^2_{L^2_s(\p\B)}.\]
We can split $ L^2_{s} (\p\B) $ into two orthogonal subspaces:
the Hardy space 
\[ H^2(\p\BB):= \Big\{ f \in L^2(\p\B) :\ \text{for a.e. $q\in\p\B,$}\  f(q)=\sum_{n\ge 0}q^na_n \ \text{with}\ \{a_n\}\in \ell^2(\N, \HH) \Big\}\]
and its orthogonal complement (with respect to the inner product) in $ L^2_{s} (\p\B)$
\[ H^2_-(\p \BB) :=  \Big\{ f \in L^2(\p\B) :\ \text{for a.e. $q\in\p\B,$}\  f(q)=\sum_{n< 0}q^na_n \ \text{with}\ \{a_{-n}\}\in \ell^2(\N, \HH) \Big\}.\]

In what follows we will denote by $\mathbb{P}_+$ and $\mathbb P _-$ the orthogonal projections from $ L^2_{s}(\p\B)$ onto $H^2(\p\B)$ and $H^2_-(\p\B)$ respectively.

Since slice regular functions on $\B$ are characterized by having a power series expansion converging in $\B$, it is natural to identify
the space $H^2(\p\B)$ with the Hardy space of the unit ball $H^2(\B)$.
Moreover, since $L^{\infty}_{s}(\p\B) \subset  L^2_{s}(\p\B),$
the space $H^{\infty}(\p\B)= L^{\infty}_{s}(\p\B)\cap H^{2}(\p\B)$ of slice $L^{\infty}$ (equivalence classes of) functions belonging to the span of $\{q^n \}_{n\ge 0}$ can be identified with the space of bounded slice regular functions $H^{\infty}(\B)$.

For any $0< p \le +\infty$ a notion of quaternionic $H^p$ space is given in \cite{hardy}. In particular in \cite{hardy} it is proven that the radial limit of slice regular functions in $H^p$ exists on each slice along almost any radius and it belongs to $L^p(\p\B_I)$ for any $0< p \le +\infty$. 
Taking into account that functions in $H^p(\B)$ satisfy the Representation Formula \eqref{repfor}, we easily conclude that their (a.e.) radial limit belongs in fact to the space $ L^p_{s}(\p\B)$. 
In particular, for any $1\le p \le +\infty$, identifying each function with its radial limit, the space $H^p(\B)$ can be viewed as a subspace of $L^p(\p\B)$ , in the sequel denoted by $H^p(\p\B)$.
Another fact about functions in $H^p(\B)$ that we will use in the sequel is that the radial limit of a function that does not vanish identically is almost everywhere nonvanishing (see again \cite{hardy}).

Our next goal is to show that the dual space of slice $L^1$ functions is the space of slice $L^{\infty}$ functions. To this aim we first need a couple of preliminary results.

\begin{pro}\label{density}
The space $L^2_s(\p\B)$ is a dense subspace of $L^1_s(\p\B)$
\end{pro}
\begin{proof}
Cauchy-Schwarz inequality implies that as in the complex case, $L^2_s(\p\B) \subset L^1_s(\p\B)$.
To prove the density, let $f \in L^1_s(\p\B)$ and fix $I\in \s$. 
The restriction $f_I$ of $f$ to $\p\B_I$ belongs to $L^1(\p\B_I)$.
Then, if $J\perp I$, and $f_I$ decomposes on $\p\B_I$ as $f_I=F^I+G^IJ$ with $F^I,G^I:\p\B_I \to L_I$, thanks to the orthogonality of $I$ and $J$ we have that both $F^I$ and $G^I$ belong to the complex space $L^1(\p\B_I,L_I)$.
Using the density of the $L^2$ space of the complex unit circle in the $L^1$ space of the complex unit circle, we find two sequences $\{F^I_n\},\{G^I_n\} \in L^2(\p\B_I, L_I)$ converging in $L^1$ norm respectively to $F^I$ and $G^I$. Therefore the sequence $f_{I,n}:=F^I_n+G^I_nJ$ converges in $L^1(\p\B_I)$
to $f_I$. If we consider the extension $f_n:=\ext f_{I,n}$ of $f_{I,n}$ to a slice function defined on $\p\B$, recalling formula \eqref{repfor}, we get that $f_n\in L^2_s(\p\B)$ for any $n \in \N$ and since $f=\ext(f_I)$, and $\ext$ is a linear operator,
\begin{align*}
\|f_n-f\|_{L^1(\p\B)}&
=\int_{\s} d\sigma(J) \frac{1}{2\pi} \int_0^{2\pi}|\ext(f_{I,n}-f_I)(e^{Jt})|dt \\
&\le \int_{\s}d\sigma(J)\frac{1}{2\pi}\int_0^{2\pi}\left(|(f_{I,n}-f_I)(e^{It})|+
|(f_{I,n}-f_I)(e^{-It})|\right)dt\\
&=\frac{1}{2\pi}\int_0^{2\pi}|(f_{I,n}-f_I)(e^{It})|dt+\frac{1}{2\pi}\int_0^{2\pi}|(f_{I,n}-f_I)(e^{-It})|dt\\
&=
2\|f_{I,n}-f_I\|_{L^1(\p\B_I)}.
\end{align*}
Hence we conclude that $f_n$ converges in $L^1(\p\B)$ to $f$.  
\end{proof}

The following result is a quaternionic version of the classical Hahn-Banach Theorem.
\begin{teo}\label{hahn}
Let $X$ be a quaternionic normed right linear space and let $Y$ be a subspace of $X$. Then for any bounded right linear opeartor $\lambda$ in the dual space $Y^*$ there exists a bounded right linear operator $\Lambda$ in $X^*$ such that $\Lambda_{|_Y}=\lambda$ and such that $\|\Lambda\|_{X^*}=\|\lambda\|_{Y^*}$.
\end{teo}
\noindent
The first part of the statement is proven in \cite{Brackx}. 
The complete result can be obtained adapting the proof of Theorem III.6 and of the following Corollary 1 in \cite{ReedSimon} from the complex case to the quaternionic case (i.e. considering the decomposition of the considered operator with respect to the real part and to 3 imaginary units).

We can finally prove the following natural relation. 
\begin{teo}\label{duals}
The dual space $(L^1_{s}(\p\B))^*$ is isometrically isomorphic to $L^{\infty}_{s}(\p\B)$. 
\end{teo}
\begin{proof}
On the one side, for any $\psi\in L^{\infty}_s(\p\B)$,  the functional $T_\psi(\cdot ):= \langle \cdot , \psi \rangle_{L^2_s(\p\B)}$, is a bounded right linear functional on $L^1_s(\B)$: for any $f\in L^1_s(\p\B)$ 
\[ |T_\psi (f)|\le \int_{\p\B}|\psi(q)||f(q)|d\Sigma\le  \|\psi\|_{L^\infty(\p\B)}\|f\|_{L^1(\p\B)},\]
i.e. $\|T_\psi\|_{\mathcal B (L^1_s(\p\B))}\le \|\psi\|_{L^\infty(\p\B)}$.

On the other side, let $\Lambda$ be a bounded right linear operator on $L^1_s(\p\B)$, and let 
$\lambda$ be its restriction to $L^2_s(\p\B)$. Since $L^2_s(\p\B)$ is a quaternionic Hilbert space we get that $\lambda \in (L^2_s(\p\B))^*=L^2_s(\p\B)$, namely that there exists $\varphi\in L^2_s(\p\B)$ such that 
\[\lambda(f)=T_\varphi(f)=\langle f, \varphi  \rangle_{L^2_s(\p\B)}\]
for any $f\in L^2_s(\p\B)$.
Thanks to the quaternionic Hahn-Banach Theorem \ref{hahn} and Proposition \ref{density}, we get that we can uniquely extend $\lambda$ to $L^1_s(\p\B)$, i.e. that 
$\Lambda (g)=T_\varphi(g)=\langle g, \varphi  \rangle_{L^2_s(\p\B)}$ for any $g\in L^1_s(\p\B)$.\\
To conclude, we are left to show that $\varphi \in L^{\infty}(\p\B)$.
We will proceed by contradiction. Suppose that for any $N>0$ there exists a measurable subset $E_N\subseteq \p\B$ such that $\Sigma(E_N)=\varepsilon_N>0$ and $|\varphi(q)|>N$ for any $q\in E_N$. Then, for any $N>0$ consider the function $\varphi_N:=\chi_{E_N}\varphi|\varphi|^{-1}$. Now $|\varphi_N|=|\chi_{E_N}|=\chi_{E_N}$, so $\varphi_N\in L^{\infty}(\p\B) \subset L^2(\p\B)=\left(L^2(\p\B)\right)^*$ since $L^2(\p\B)$ is a quaternionic Hilbert space. Hence, on the one hand, using Cauchy-Schwarz inequality we can write
\begin{align}\label{aprile} \nonumber
\left|\langle \varphi, \varphi_N \rangle_{L^2(\p\B)}\right|&\le \|\varphi\|_{L^2(\p\B)}\cdot\|\varphi_N\|_{L^2(\p\B)}=\|\varphi\|_{L^2(\p\B)}\int_{\p\B}|\varphi_N|^2d\Sigma\\
&=\|\varphi\|_{L^2(\p\B)}\int_{\p\B}\chi_{E_N}d\Sigma=\|\varphi\|_{L^2(\p\B)}\varepsilon_N.
\end{align} 
On the other hand
\begin{align}\label{maggio} 
&\left|\langle \varphi, \varphi_N \rangle_{L^2(\p\B)}\right| =\left|\int_{\p\B} \chi_{E_N} |\varphi|^{-1}\overline{\varphi}\varphi d\Sigma\right|
=\int_{E_N}|\varphi|d\Sigma>N\varepsilon_N .
\end{align}
Combining inequalities \eqref{aprile} and \eqref{maggio} we get that 
\[\|\varphi\|_{L^2(\p\B)}>N.\]
Since $\varphi$ is bounded in $L^2$ norm and $N$ is arbitrarily large, we get a contradiction, thus proving that $\varphi \in L^{\infty}(\p\B)$.

%
\end{proof}

Another property that we will need in the sequel is the following.
\begin{pro}\label{normainfinito}
A function $\varphi$ belongs to $L^{\infty}_s(\p\B)$ if and only if its regular conjugate $\varphi^c$ does. Moreover the two norms coincide, $\|\varphi\|_{L^{\infty}(\p\B)}=\|\varphi^c\|_{L^{\infty}(\p\B)}$.
\end{pro}
\noindent The proof exploits the fact that we are considering {slice} functions admitting a power series expansion, and follows the same lines than the proof in the slice regular case, see Proposition 5 and Corollary 1 in \cite{BlochDGS}.

We can now give the notion of slice $BMO$ space, i.e. the space of slice functions $f:\p\B \to \HH$ with bounded mean oscillation. 
\begin{defn}
Let $f\in  L^1_{s}(\p\B)$ and, for any interval $a=(\alpha, \beta)$ of $\rr$ such that $|a|:=|\beta-\alpha|\le 2\pi$, denote by $f_{I,a}$ the average value of $f$ on the arc $(e^{\alpha I}, e^{\beta I})\subseteq \p \B_I$,
\[f_{I,a}=\frac{1}{|a|}\int_{a}f(e^{\theta I})d\theta.\]
We say that $f\in BMO(\p \B_I)$ if 
\[\| f\| _{BMO(\p \B_I)}:=\sup_{{a\subset \rr, \ |a| \le 2\pi}}\left\{\frac{1}{|a|}\int_{a}|f(e^{\theta I})-f_{I,a}|d\theta\right\}<+\infty.\] 
We say that $f \in BMO (\p \B)$ if 
\[\| f\| _{BMO(\p \B)}:=\sup_{I\in \s}\| f\| _{BMO(\p\B_I)}<+\infty.\] 
\end{defn}
The space of slice regular functions on $\B$ with bounded mean oscillation has been introduced in \cite{carleson} and there it is denoted by $BMOA(\B)$.
In this paper, since in fact $BMOA(\B)$ coincides with the intersection of $BMO(\p\B)$ with the space $H^1(\B)$, viewed as a subspace of $L^1_s(\p\B)$, we will denote it as $BMOA(\p\B)$. 

\noindent
Taking into account the Representation Formula \eqref{repfor} it is possible to show the following result.
\begin{pro}\label{BMOslice}
Let $f\in L_{s}^1(\p\B)$. Then $f\in BMO(\p\B)$ if and only if $f\in BMO(\p\B_I)$ for every $I\in \s$.
\end{pro} 
\noindent See \cite{carleson} for a proof in the case of slice regular functions.  

As in the complex case (see, e.g., \cite{garnett}) we have a characterization of slice $BMO$ functions in terms of slice $L^{\infty}$ functions.
\begin{pro}\label{decom}
A slice $L^1$ function $f$ belongs to $BMO(\p\B)$ if and only if there exist $\varphi, \psi \in L_{s}^{\infty}(\p\B)$ such that
$f$ admits the representation
\begin{equation}\label{bmorep}
f=\varphi+\mathbb P_+ \psi.
\end{equation}
\end{pro}
\begin{proof} 
Consider the splitting of $f$ on the slice $L_I$ with respect to $J\in\s$, $J$ orthogonal to $I$, $f=F+GJ$ for some functions $F,G:\p\B_I \to L_I$. For any real interval $a=(\alpha, \beta)$, with $|a|\le 2\pi$, denote by $F_{a}$ and $G_a$ the average values of $F$ and $G$ on the arc $(e^{I\alpha}, e^{I\beta})$. Thanks to the orthogonality of $I$ and $J$ we have that 
\begin{align*}
\|f\|_{BMO(\p\B_I)}=\sup_{{a\subset \rr, \ |a| \le 2\pi}}\left\{\frac{1}{|a|}\int_{a}\left(|F(e^{\theta I})-F_{a}|^2+|G(e^{\theta I})-G_{a}|^2\right)^{1/2}d\theta\right\}.\\
\end{align*}
Since the quantity on the right side is greater or equal than both $\|F\|_{BMO(\p\B_I)}$ and $\|G\|_{BMO(\p\B_I)}$, and is smaller than the sum $\|F\|_{BMO(\p\B_I)}+\|G\|_{BMO(\p\B_I)}$, we have that $f$ belogns to $BMO(\p\B_I)$ (and hence to $BMO(\p\B)$) if and only if 
both $F$ and $G$ belong to the complex $BMO$ space on the unit circle $\p\B_I$.
A consequence of the classical Fefferman Theorem (the results are in fact equivalent, see \cite{garnett}) yields that this is equivalent to the fact that $F$ and $G$ admit a representation of the form
\[F=\varphi_1+\mathbb P^I_+ \psi_1, \quad G=\varphi_2+\mathbb P^I_+ \psi_2\]
where $\mathbb P^I_+$ denotes the orthogonal projection from the complex space $L^2(\p\B_I, L_I)$ to the complex Hardy space $H^2(\p\B_I)$ and $\varphi_1,\varphi_2,\psi_1,\psi_2\in L^{\infty}(\p\B_I, L_I)$.
Therefore $f\in BMO(\p\B_I)$ if and only if 
\[f_I=F+GJ=\varphi_1+\mathbb P^I_+ \psi_1 + (\varphi_2+\mathbb P^I_+ \psi_2)J= \varphi_1+\varphi_2J + \mathbb P^I_+( \psi_1 +\psi_2J)\]
where, thanks to the orthogonality of $I$ and $J$, $\varphi_I:=\varphi_1+\varphi_2J$ and $\psi_I:=\psi_1 +\psi_2J$ belong to $L^{\infty}(\p\B_I)$.
Extending the functions $\varphi_I$ and $\psi_I$ by means of formula \eqref{repfor}, i.e. defining 
$\varphi(x+yJ)=\ext(\varphi_I)(x+yJ)$ 
and $\psi(x+yJ)=\ext(\psi_I)(x+yJ)$, leads us to conclude: slice functions belongs to $L^{\infty}(\p\B)$ if and only if they belong to $L^{\infty}(\p\B_I)$ on one and hence on each slice $\p\B_I$ (see Proposition \ref{sliceLP}).
\end{proof}

\section{Hankel operators}\label{main}
Given any quaternion valued sequence $\alpha:\N\to\HH$ we can define
the Hankel operator $\Gamma_\alpha$, acting on $\HH$ valued sequences $v=(v_j)_{j=0}^{+\infty}$ as 
\[(\Gamma_\alpha v)(j)=\sum_{k=0}^{\infty}\alpha(j+k)v(k),\]
i.e. by matrix multiplication by the infinite matrix $M_\alpha=[\alpha(j+k)]_{j,k=0}^{+\infty}$ with constant antidiagonals.
As we said in the introduction, 
Nehari's problem concerns the characterization of {\em bounded} Hankel operators.  
Let us begin by a reformulation of Nehari Theorem in the quaternionic setting, involving slice $L^{\infty}$ functions (compare with Theorem 1.1 in \cite{carleson}).

\begin{teo}\label{nehari1}
Let $\alpha=\{\alpha_n\}_{n\in \N}\subset \HH$ be a quaternion valued sequence. Then the operator $\Gamma_{\alpha }$ is bounded on $\ell^2(\N, \HH)$ if and only if there exists $\psi \in L_{s}^{\infty}(\p\B)$ such that $\widehat{\psi}(m)=\overline{\alpha}_m$ for any $m\ge 0$.
In this case 
\begin{equation}\label{normaGamma}
\begin{aligned}
\inf\{\|\psi\|_{L^{\infty}(\p\B)} \,& : \, \psi \in L^{\infty}_s(\p\B), \,  \widehat{\psi}(m)= \overline{\alpha}_m, \, m\ge 0\}\le \|\Gamma_{\alpha}\|_{\mathcal{B}(\ell^2(\N, \HH))}\\
&\le 2 \inf\{\|\psi\|_{L^{\infty}(\p\B)} \, : \, \psi \in L^{\infty}_s(\p\B), \, \widehat{\psi}(m)= \overline{\alpha}_m, \, m\ge 0\}.
\end{aligned}
\end{equation}
\end{teo}

\begin{proof}

Let $\varphi$ be the generating function of the sequence $\overline{\alpha}$, $\varphi(q)=\sum_{n \in \N}q^n\overline\alpha_n$.
\noindent Theorem \ref{passouno} (proved in \cite{carleson}) states that $\Gamma_{\alpha}$ is bounded if and only if $\varphi$ belongs to the space $BMOA(\p\B)$ of slice regular $BMO$ functions. 
Recalling Proposition \ref{bmorep}, we get that $\varphi$ is in $BMOA(\p\B)$ if and only if there exists $\psi \in  L^{\infty}_{s}(\p\B)$ such that $\mathbb{P}_+\psi=\varphi$, namely such that 
$\widehat{\psi}(m)=\overline{\alpha}_m$ for any $m\ge 0$,
thus proving the first part of the statement.

Suppose now that the operator $\Gamma_\alpha$ is bounded. To estimate its norm, consider the bilinear operator $G_{\alpha}$ associated with $\Gamma_{\alpha}$, defined on the dense subset of couples of finitely supported sequences $(a=\{a_n\}_{n\in \N}, b=\{b_n\}_{n\in \N})$ in $\ell^2(\N, \HH)\times\ell^2(\N, \HH)$ by    
\[G_{\alpha}(a,b):=\left\langle b ,  \overline{\Gamma_{\alpha} a}\right\rangle_{\ell^2(\N, \HH)}=\sum_{n\ge 0 }\sum_{k\ge 0 }\alpha_{n+k} a_k b_n .\]
It is not difficult to see that $G_{\alpha}$ is bounded if and only if $\Gamma_{\alpha}$ does, and 
\[\sup_{{ c,d \in \ell^2(\N,\HH), c,d \ne 0} }\frac{\left|G_\alpha(c,d)\right|}{\|c\|_{\ell^2(\N, \HH)}\cdot\|d\|_{\ell^2(\N,\HH)}}=\|\Gamma_{\alpha}\|_{\mathcal B(\ell^2(\N, \HH))}.\]
Let $f(q):=\check{a}(q)=\sum_{n\ge 0}q^na_n$ and $g(q):=\check{b}(q)=\sum_{n\ge 0}q^nb_n$ be polynomials in $H^2(\p\B)$. Then
\begin{equation}\label{minore}
G_{\alpha}(a,b)=\sum_{n\ge 0 }\sum_{j\ge 0 }\alpha_{j} a_{j-n}b_n =\sum_{j\ge 0 }\alpha_{j}\sum_{n= 0 }^j a_{j-n} b_n =\left\langle  f\star g, \varphi\right\rangle_{ L^2_{s} (\p\B) }.\end{equation}
Moreover, the previous equality holds true for any $\psi\in L^{\infty}_s(\p\B)$ such that $\mathbb{P}_+\psi=\varphi$. Thus, for such a $\psi$, we can write 
\begin{align*}
|G_{\alpha}(a,b)|&=\left|\left\langle f\star g, \psi \right\rangle_{ L^2_{s} (\p\B) }\right|= \left| \int_{\p\B}\overline{\psi(q)}(f\star g)(q)d\Sigma(q)\right| 
\le \int_{\s}d\sigma(I)\int_{0}^{2\pi}\left|\overline{\psi(e^{It})}\right|\left|f\star g(e^{It})\right|dt\\
&\le \|\psi\|_{L^{\infty}(\p\B)} \int_{\s}d\sigma(I)\int_{0}^{2\pi}\left|f(e^{It})\right|\left|g(e^{Jt})\right|dt
\end{align*}
where $J=(f(e^{It}))^{-1}If(e^{It})$ is determined in view of Proposition \ref{trasf} (the fact that $f\in H^2(\p\B)$ guarantees that if $f\not \equiv 0$, then it is non vanishing almost everywhere at the boundary, see \cite{hardy}).
Using Representation Formula \eqref{repfor}, we have 
\[|g(e^{Jt})|\le |g(e^{It})|+|g(e^{-It})|,\]
so that
\begin{align*}
|G_{\alpha}(a,b)| \le \|\psi\|_{L^{\infty}(\p\B)}&\left( \int_{\s}d\sigma(I)\int_{0}^{2\pi}\left|f(e^{It})\right|\left|g(e^{It})\right|dt 
 +\int_{\s}d\sigma(I)\int_{0}^{2\pi}\left|f(e^{It})\right|\left|g(e^{-It})\right|dt\right).
\end{align*}
By Cauchy-Schwarz inequality we get then
\begin{align}\label{disu}
|G_{\alpha}(a,b)| &\le 2\|\psi\|_{L^{\infty}(\p\B)}\cdot\|f\|_{L^2_s(\p\B)}\cdot \|g\|_{L^2_s(\p\B)} 
\end{align}
 where we used the fact 
 that if $\tilde g (q)=g(\bar q)$ then 
 $\|\tilde g \|_{L^2_s(\p\B)}=\|g\|_{L^2_s(\p\B)}$ for any $g\in H^2(\p\B)$.
By density of finitely supported sequences in $\ell^2(\N, \HH)$, we obtain 
\[\|\Gamma_{\alpha}\|_{\mathcal{B}(\ell^2(\N, \HH))}=\|G_{\alpha}\|_{\mathcal{B}(\ell^2(\N, \HH)\times \ell^2(\N, \HH))}\le 2\|\psi\|_{L^{\infty}(\p\B)}.\]
The fact that $\psi$ is an arbitrary element of $L^{\infty}_s(\p\B)$ such that $\mathbb P_+ \psi= \varphi$ yields that 
\begin{align*}
\|\Gamma_{\alpha}\|_{\mathcal{B}(\ell^2(\N, \HH))}\le 2\inf\{\|\psi\|_{L^{\infty}(\p\B)}  : \psi \in L^{\infty}_s(\p\B), \,  \widehat{\psi}(m)= \overline{\alpha}_m, \, m\ge 0\}.
\end{align*}

On the other hand, since $\varphi(q)=\sum_{n\ge 0}q^n\overline{\alpha}_n$ belongs to $BMOA(\p\B)$, which, as stated by Theorem 5.6 in \cite{carleson}, is the dual space of $H^1(\p\B)$, we have that the linear operator $\Lambda_{\alpha}: H^1(\p\B) \to \HH $ given by
\[
\Lambda_\alpha h :=\langle h, \varphi \rangle_{L^2_s(\p\B)}=\sum_{n\ge 0} \alpha_{n}\widehat{h}(n) 
,\]
is well defined for any $h(q)=\sum_{n\ge 0}q^n\widehat{h}(n)\in H^1(\p\B)$.
As proven in \cite{carleson}, we can identify $H^1(\p\B)$ with $H^2(\p\B)\star H^2(\p\B)+H^2(\p\B)\star H^2(\p\B)$, where the norm is defined as 
\begin{align*}
&\|h\|_{H^2(\p\B)\star H^2(\p\B)+H^2(\p\B)\star H^2(\p\B)}\\
&=\inf\Big\{\sum_{t=1,2}\|f_t\|_{H^2(\p\B)}\cdot\|g_t\|_{H^2(\p\B)} : h=\sum_{t=1,2}f_t\star g_t, f_t,g_t \in H^2(\p\B) \Big\}.
\end{align*}
Hence if $h\in H^1(\p\B)$ decomposes as $h=f_1\star g_1+f_2\star g_2$, with $f_t(q)=\sum_{n\ge 0}q^n(a_t)_n, g_t(q)=\sum_{n\ge 0}q^n(b_t)_n \in H^2(\p\B)$, for $t=1,2$, we can write
\begin{align*}
\Lambda_\alpha (h)&= \sum_{n\ge 0} \alpha_{n}\sum_{k=0}^n (a_1)_k (b_1)_{n-k}+ \sum_{n\ge 0}\alpha_{n}\sum_{k=0}^n (a_2)_k (b_2)_{n-k }= G_{\alpha}(a_1, b_1)+G_{\alpha}(a_2,b_2).
\end{align*}
If $\Gamma_{\alpha}$ (and hence $G_\alpha$) is bounded, then
\begin{align*}
|\Lambda_\alpha (h)|&\le \|G_{\alpha}\|_{\mathcal{B}(\ell^2(\N, \HH)\times \ell^2(\N, \HH))} \left(\|a_1\|_{\ell^2(\N, \HH)}\cdot \|b_1\|_{\ell^2(\N, \HH)} + \|a_2\|_{\ell^2(\N, \HH)}\cdot \|b_2\|_{\ell^2(\N, \HH)}\right)\\ 
&= \|G_{\alpha}\|_{\mathcal{B}(\ell^2(\N, \HH)\times \ell^2(\N, \HH))} \left(\|f_1\|_{H^2(\p\B)}\cdot\|g_1\|_{H^2(\p\B)}+\|f_2\|_{H^2(\p\B)}\cdot\|g_2\|_{H^2(\p\B)}\right).
\end{align*}
Since the decomposition $h=f_1\star g_1+f_2\star g_2$ is arbitrary, taking  the infimum on all possible decompositions of $h$ in $H^2(\p\B)\star H^2(\p\B)+H^2(\p\B)\star H^2(\p\B)$, we get  
\[|\Lambda_\alpha(h)|\le \|G_{\alpha}\|_{\mathcal{B}(\ell^2(\N, \HH)\times \ell^2(\N, \HH))} \|h\|_{H^2(\p\B)\star H^2(\p\B)+H^2(\p\B)\star H^2(\p\B)} \]
and hence that
\begin{equation}\label{normaminore}
\|\Lambda_\alpha\|_{\mathcal B (H^2(\p\B)\star H^2(\p\B)+H^2(\p\B)\star H^2(\p\B))}\le \|\Gamma_{\alpha}\|_{\mathcal{B}(\ell^2(\N, \HH))}.
\end{equation}
Since $H^2(\p\B)\star H^2(\p\B)+H^2(\p\B)\star H^2(\p\B)=H^1(\p\B)$ is a linear subspace of $L^1_s(\p\B)$, by the quaternionic Hahn-Banach Theorem \ref{hahn} we can extend $\Lambda_\alpha$ to a bounded linear operator $\tilde \Lambda$ on $L^1_s(\p\B)$ with same norm.
Recalling Theorem \ref{duals}, we get then that there exists $\psi\in L^{\infty}_s(\p\B)$ such that 
$\tilde \Lambda (f) = \langle f, \psi  \rangle_{L^2_s(\p\B)}$ for any $f\in L^1(\p\B)$ and such that
\[\|\Lambda_\alpha\|_{\mathcal B (H^2(\p\B)\star H^2(\p\B)+H^2(\p\B)\star H^2(\p\B))}= \|\tilde \Lambda \|_{\mathcal B (L^1(\p\B))}=\|\psi\|_{L^{\infty}(\p\B)}.\]

\noindent Moreover, since $\tilde \Lambda_{|_{H^2(\B)\star H^2(\B)+H^2(\B)\star H^2(\B)}}=\Lambda_\alpha$, we get that $\mathbb{P}_+\psi=\varphi$ and hence 
we conclude 
\[\inf\{\|\psi\|_{L^{\infty}(\p\B)} : \psi \in L^{\infty}_s(\p\B), \,  \widehat{\psi}(m)= \overline{\alpha}_m, \,  m\ge 0\}\le \|\Gamma_{\alpha}\|_{\mathcal{B}(\ell^2(\N, \HH))}.\]
%
\end{proof}


Hankel operators admit also a realization as operators from the Hardy space $H^2(\p\B)$ to its orthogonal complement $H^2_- (\p\B)$.
Let $\varphi\in  L^2_{s} (\p\B) $ have power series expansion $\varphi(q)=\sum_{n\in \Z}q^n\widehat{\varphi}(n)$. 
On the dense subset of polynomials in $H^2(\p\B)$, we can define the operator associated with $\varphi$ 
\[H_{\varphi}: H^2(\p\B) \to H^2_- (\p\B)\]
as
\begin{equation}\label{projection}
H_{\varphi}f=\mathbb{P}_- ( \varphi \star f ).
\end{equation} 
If 
we 
do explicit computations we get
\begin{equation}\label{proj2}
H_{\varphi} f= \mathbb P_- \Big(\sum_{n\in \Z}q^n\sum_{k \in \Z}\widehat{\varphi}(n-k)\widehat{f}(k)\Big)=\sum_{n < 0}q^n\sum_{k \ge 0}\widehat{\varphi}(n-k)\widehat{f}(k).
\end{equation}
Hence the matrix associated with $H_\varphi$ with respect to the basis $\{q^n\}_{n\ge 0}$ of $H^2(\p\B)$ and $\{\bar q^{n}\}_{n\ge 0}$ of $H^2_-(\p\B)$,
 is the Hankel matrix $(\widehat{\varphi}(n-k))_{n < 0,k \ge 0}$.

With this in mind, the characterization of bounded Hankel operators, acting on $H^2(\p\B)$, can be given in the following way. 


\begin{teo}\label{nehari2}
Let $\varphi(q)=\sum_{n\in \Z}q^n \widehat{\varphi}(n) \in  L^2_{s} (\p\B) $. The following conditions are equivalent: 
\begin{itemize}
\item[(1)] $H_\varphi$ is bounded on $H^2(\p\B)$;
\item[(2)] there exists $\psi \in  L^{\infty}_{s}(\p \B)$ such that $\widehat{\psi}(m)=\overline{\widehat{\varphi}(m)}$ for any $m<0$;
\item[(3)] $\mathbb P_-\varphi \in BMO(\p\B) $.
\end{itemize}
If one of these conditions holds, then

\begin{equation}\label{normaHfi}
\begin{aligned}
\inf\{\|\psi\|_{L^{\infty}(\p\B)} \, : \,\psi \in L^{\infty}_s(\p\B) ,\, &\widehat{\psi}(m)= \overline{\widehat{\varphi}(m)},\, m < 0\} \le \|H_\varphi\|_{\mathcal B (H^2(\p\B))}\\
&\le 2 \inf\{\|\psi\|_{L^{\infty}(\p\B)}  : \psi \in L^{\infty}_s(\p\B) ,\, \widehat{\psi}(m)= \overline{\widehat{\varphi}(m)}, \, m < 0\}.
\end{aligned}
\end{equation}
\end{teo}
\begin{proof}
Consider the sequence $\widehat{\varphi}:=\{\widehat{\varphi}(-n)\}_{n\ge 0} \in \ell^2(\N, \HH)$. 
Equation \eqref{proj2} implies that $H_\varphi$ is bounded on $H^2(\p\B)$ if and only if the operator $\Gamma_{\widehat{\varphi}}$ is bounded on $\ell^2(\N, \HH)$.
The equivalence of statements $(1)$ and $(2)$ is therefore a consequence of Theorem \ref{nehari1}: $\Gamma_{\widehat{\varphi}}$ is bounded if and only if there exists $\psi \in  L^{\infty}_{s}(\p\B)$ such that $\widehat{\psi}(m)=\overline{\widehat{\varphi}(-m)}$ for any $m \ge 0$ which, considering $\psi(\bar q)$, is equivalent to condition $(2)$. 

The equivalence of $(2)$ and $(3)$ can be proven as follows. First, Proposition \ref{decom} states that $(3)$ is equivalent to the existence of a function $\psi \in  L^{\infty}_{s}(\p \B)$ such that $\mathbb P_-\psi=\mathbb P_-\varphi$. Then, since  $\psi \in L^{\infty}_s(\p\B)$ if and only if $\psi^c \in L^{\infty}_s(\p\B)$ (see Proposition \ref{normainfinito}) and the projection operator $\mathbb{P}_-$ commutes with the regular conjugation, we conclude that $\mathbb P_-(\psi ^c) =\mathbb P_-(\varphi^c)$,  i.e. that $\psi^c$ is the function that realizes condition $(2)$.

Inequalities \eqref{normaHfi} are a direct application of inequalities \eqref{normaGamma} to the operator $\Gamma_{\widehat \varphi}$. 
\end{proof}
Equivalence of conditions $(1)$ and $(2)$ in particular says that $\varphi$ in $L^2_s(\p\B)$ is such that $H_\varphi$ is bounded if and only if there exists $\psi\in L^{\infty}_s(\p\B)$ such that $H_\varphi=H_\psi$. 
In fact equation \eqref{proj2} implies that $H_\varphi$ depends only on the negative coefficients of $\varphi$.  
Thus the class of bounded Hankel operators on $H^2(\p\B)$ is covered by operators of the form $H_\varphi$ with a bounded symbol $\varphi \in  L^{\infty}_{s}(\p\B)$. 
We can finally prove the announced result relating the norm of a Hankel operator with the $L^{\infty}$-distance of its symbol from the space of bounded slice regular functions.
\begin{teo}\label{43}
Let $\varphi \in  L^{\infty}_{s}(\p\B)$. Then 
\begin{equation}\label{inf}
\|H_\varphi\|_{\mathcal B (H^2(\p\B))}=\inf\{\|\varphi-f\|_{L^{\infty}(\p\B)} \, : \, f \in H^{\infty}(\p \B)
\}.
\end{equation}
\end{teo}

\begin{proof}
Let $\psi \in  L^{\infty}_{s}(\p\B)$ be such that $\mathbb P_-\psi= \mathbb P_- \varphi$. Then $H_\varphi=H_\psi$ and
\[\|H_\varphi\|_{\mathcal B (H^2(\p\B))}=\|H_\psi\|_{\mathcal B (H^2(\p\B))}=\sup_{f\in H^2(\p\B), f\not\equiv 0}\frac{\|H_\psi(f)\|_{ L^2_{s}(\p\B)}}{\|f\|_{ L^2_{s}(\p\B)}}.\]
Since $H_\psi$ is bounded, we can express it as a projection (extending definition \eqref{projection} to the whole $H^2(\p\B)$), thus obtaining
\begin{align}\label{disuconj}
\|H_\psi (f)\|_{ L^2_{s}(\p\B)}&=\|\mathbb P_-(\psi\star f )\|_{ L^2_{s}(\p\B)}\le \|\psi\star f\|_{ L^2_{s}(\p\B)}=\|(\psi\star f)^c\|_{ L^2_{s}(\p\B)}\\ \nonumber
&=\|f^c\star \psi^c\|_{ L^2_{s}(\p\B)}\le \|f^c\|_{ L^2_{s}(\p\B)}\|\psi^c\|_{L^{\infty}(\p\B)},
\end{align}
where we used Proposition \ref{scambio} and where the last inequality follows from the integral form \eqref{normaint} of the slice $L^2$ norm and from Proposition \ref{trasf}: if $f\in H^2(\p\B)$ and $f\not \equiv 0$, then $f^c \in H^2(\p\B)$ and it is almost everywhere non-vanishing on $\p\B$ so we can write
\[\|f^c\star \psi^c\|_{ L^2_{s}(\p\B)}=\int_{\p\B}|f^c(e^{It})||\psi^c(f^c(e^{It})^{-1}e^{It}f^c(e^{It}))|d\Sigma(e^{It})\le\int_{\p\B}|f^c(e^{It})|d\Sigma(e^{It})\|\psi^c\|_{L^{\infty}(\p\B)}.\] 
Recalling that for any $f\in H^2(\p\B)\subset L^2_s(\p\B)$, $\|f\|_{L^2_s(\p\B)}=\|f^c\|_{L^2_s(\p\B)}$ and that, thanks to Proposition \ref{normainfinito}, for any $\psi \in L^{\infty}_{s}(\p\B)$, $\|\psi\|_{L^{\infty}(\p\B)}=\|\psi^c\|_{L^{\infty}(\p\B)}$,
we get
\begin{equation}\label{18}
\|H_\psi (f)\|_{ L^2_{s}(\p\B)}\le \|f\|_{ L^2_{s}(\p\B)}\|\psi\|_{L^{\infty}(\p\B)},
\end{equation}
and hence that 
\[\|H_\varphi\|_{\mathcal B (H^2(\p\B))}\le \|\psi\|_{L^{\infty}(\p\B)}.\]
Since $\psi$ was an arbitrary element of $ L^{\infty}_{s}(\p\B)$ such that $\mathbb P_-\psi= \mathbb P_- \varphi$,  we obtain that
\[\|H_\varphi\|_{\mathcal B (H^2(\p\B))}\le\inf\{\|\psi\|_{L^{\infty}(\p\B)} \, : \, \psi \in L^{\infty}_s(\p\B), \, \widehat{\psi}(m)= \widehat{\varphi}(m),  m < 0\}\]
and hence, recalling the first inequality in \eqref{normaHfi}, that

\[\|H_\varphi\|_{\mathcal B (H^2(\p\B))}=\inf\{\|\psi\|_{L^{\infty}(\p\B)} \, :  \, \psi \in L^{\infty}_s(\p\B), \, \widehat{\psi}(m)= \widehat{\varphi}(m),  m < 0\}\]

\noindent which is a reformulation of equality \eqref{inf}: on the one hand $\varphi-f \in L^{\infty}_s(\p\B)$ and
$\mathbb P_-(\varphi - f)=\mathbb P_- \varphi$ for any $f\in H^{\infty}(\p\B)$; on the other hand $\varphi-\psi \in H^{\infty}(\p\B)$ for any $\psi \in L^{\infty}_s(\p\B)$ such that $\mathbb{P}_-\psi= \mathbb{P}_-\varphi$.

\end{proof}

To prove that the infimum in equation \eqref{inf} is attained for any $\varphi \in L^{\infty}_{s}(\p\B)$ we need the following version of Montel Theorem for slice regular functions.
\begin{teo}\label{montel}
Let $\{f_n\}$ be a uniformly bounded family of slice regular functions on $\B$. Then $\{f_n\}$ has a subsequence convergent uniformly on compact sets to a slice regular function. 
\end{teo}
\begin{proof}
Let $I\in \s$ be any imaginary unit and let $J\in \s$, $J$ orthogonal to $I$. For any $n$, consider the restriction $f_{I,n}$ of $f_n$ to $\B_I$. Thanks to the Splitting Lemma (see \cite{libroGSS}), for any $n$, there exist two holomorphic functions $F_{I,n},G_{I,n}$ such that $f_{I,n}=F_{I,n}+G_{I,n}J$. Since for any $n$ $|f_{I,n}|^2=|F_{I,n}|^2+|G_{I,n}|^2$ we get that both $F_{I,n}$ and $G_{I,n}$ are uniformly bounded. 
Thanks to the classical Montel Theorem, up to subsequences both $F_{I,n}$ and $G_{I,n}$ converge uniformly on compact sets to holomorphic functions $F$ and $G$ respectively. Hence, up to subsequences, $f_{I,n}$ converges uniformly on compact sets to $f_I:=F+GJ$. Extending $f_I$ by means of the representation formula \eqref{repfor} we get that,  up to subsequences, $f_n$ converges uniformly on compact sets to $f:=\ext(f_I)$.  The uniform convergence guarantees that $f$ is slice regular.    
\end{proof}

Now we can show that the infimum in equation \eqref{inf} is in fact a minimum.
\begin{pro}
Let $\varphi \in L^{\infty}_{s}(\p\B)$. Then there exists $g\in H^{\infty}(\p\B)$ realizing the distance of $\varphi$ from the space of bounded slice regular functions, that is such that 
\[\|H_\varphi\|_{\mathcal B (H^2(\p\B))}=\|\varphi-g\|_{L^{\infty}(\p\B)}.\]
Such a $g$ is called an {\em $L^{\infty}$-best approximation of $\varphi$ by a slice regular function}.
\end{pro}

\begin{proof}
Let $m=\inf\{\|\varphi-f\|_{L^{\infty}(\p\B)} \, : \, f \in H^{\infty}(\p \B)
\}$. Since the zero function is a competitor, we get that $m\le \|\varphi\|_{L^{\infty}(\p\B)}$. 
If $m=\|\varphi\|_{L^{\infty}(\p\B)}$, then the zero function is the desired best approximation slice regular function. Otherwise, let $\{f_n\}$ be a minimizing sequence, i.e. a sequence of functions in $H^{\infty}(\p\B)$ such that $\lim_{n\to \infty} \|\varphi-f_n\|_{L^{\infty}(\p\B)}=m$. Then, for $n$ sufficiently large, we have
\[ \|f_n\|_{L^{\infty}(\p\B)}-\|\varphi\|_{L^{\infty}(\p\B)} \le\|\varphi-f_n\|_{L^{\infty}(\p\B)}\le\|\varphi\|_{L^{\infty}(\p\B)}\]
that implies
\[\|f_n\|_{L^{\infty}(\p\B)}\le 2\|\varphi\|_{L^{\infty}(\p\B)},\]
i.e. that $\{f_n\}$ is uniformly bounded. Theorem \ref{montel} lead us to conclude.
\end{proof}
 
%
%

A natural question is then investigate the uniqueness of such a function.
As in the complex case (see \cite{peller}) we can prove the following.
\begin{teo}
Let $\varphi \in  L^{\infty}_{s}(\p\B)$ be such that $H_\varphi$ attains its norm on the unit ball of $H^2(\p\B)$, that is such that $\|H_\varphi\|_{\mathcal B (H^2(\p\B))}=\|H_\varphi g\|_{ L^2_{s} (\p\B) }$ for some $g\in H^2(\p\B)$ with $\|g\|_{ L^2_{s} (\p\B) }=1$.
Then there exists a unique $f\in H^{\infty}(\p\B)$ such that 
\[\|\varphi-f\|_{L^{\infty}(\p\B)}={\rm dist}_{L^{\infty}} (\varphi, H^{\infty}(\p\B)) .\]
\end{teo}
 \begin{proof}
Suppose without loss of generality that $\|H_\varphi\|_{\mathcal B (H^2(\p\B))}=1$, and let $f\in H^{\infty}(\p\B)$ be such that
\[\|\varphi-f\|_{L^{\infty}(\p\B)}=\|H_\varphi\|_{\mathcal B (H^2(\p\B))}=1.\]
Then, since $g$ realizes the norm of $H_{\varphi}$, and $H_{\varphi}=H_{\varphi-f}$ (since $\mathbb P_-\varphi=\mathbb P_- (\varphi -f)$), proceeding as we have done in the proof of Theorem \ref{43} to obtain  \eqref{disuconj} and \eqref{18}, we get 
\begin{align*}
1&=\|H_\varphi\|_{\mathcal B (H^2(\p\B))}=\|H_{\varphi-f}\|_{\mathcal B (H^2(\p\B))}=\|H_{\varphi - f}g\|_{L_s^2(\p\B)} =\|\mathbb{P}_-((\varphi - f)\star g)\|_{L_s^2(\p\B)}\\
&\le\| (\varphi - f)\star g \|_{L_s^2(\p\B)}
\le \|g\|_{L_s^2(\p\B)}\|\varphi - f\|_{L^{\infty}(\p\B)}=\|g\|_{L_s^2(\p\B)}=1.
\end{align*}
Therefore all inequalities are equalities, and in particular 
\[\|\mathbb{P}_-( (\varphi - f)\star g)\|_{L_s^2(\p\B)}=\| (\varphi - f)\star g\|_{L_s^2(\p\B)}\]
which means that
$ (\varphi - f)\star g \in H^2_-(\p\B)$ and that 
\[H_\varphi g=H_{\varphi- f}g= (\varphi - f)\star g.\]
Recalling that the symmetrization of a function in $H^2(\p\B)$ belongs to $H^1(\p\B)$ and that functions in $H^1(\p\B)$ are nonvanishing almost everywhere at the boundary (see \cite{hardy}), we can consider the regular reciprocal of $g$, $g^{-\star}=(g^s)^{-1}g^c$, and obtain $f$ as
\[f=\varphi - H_{\varphi}g\star g^{-\star}\]
where $ H_{\varphi}g\star g^{-\star}$ does not depend on $g$ (if $\tilde g \in H^{2}(\p\B)$ is such that 
$\|H_\varphi\|_{\mathcal B (H^2(\p\B))}=\|H_\varphi \tilde g\|_{ L^2_{s} (\p\B) }$ with $\|\tilde g\|_{ L^2_{s} (\p\B) }=1$, then, using the same arguments used for the function $g$, we get that $H_{\varphi}\tilde g\star \tilde g^{-\star}=H_{\varphi}g\star g^{-\star}=\varphi-f$).
Therefore we get that $f$ is uniquely determined by $\varphi$.

\end{proof} 


We conclude this paper with a characterization of Hankel operators involving {\em shift} operators, that reflects the one holding in the complex case, see \cite{peller}.

\noindent Let $S:  L^2_{s} (\p\B)  \to  L^2_{s} (\p\B) $ denote the bilateral shift operator, 
\[(S f)(q)=q\star f(q)=qf(q) \text{ for any }  f\in  L^2_{s} (\p\B), \]
whose adjoint is $(S^*f)(q)= \bar q \star f(q)$, 
and let $T: H^2(\p\B) \to H^2(\p\B)$ denote the right shift operator, 
\[(T f)(q)=q\star f(q)=qf(q)  \text{ for any } f\in H^2(\p\B)\]
whose adjoint is the left (or backward) shift operator introduced in \cite{milanesikrein}, $(T^*f )(q)=q^{-1}(f(q)-f(0))$.
Notice that the operators $\mathbb{P_-}$ and $S$ do commute.
\begin{teo}
Let $R:H^2(\p\B)\to H^2_-(\p\B)$ be a bounded operator. Then $R$ is a Hankel operator if and only if 
\begin{equation}\label{comm}
\mathbb P_-S R=RT.
\end{equation}  
\end{teo}
\begin{proof}
Suppose first that $R$ is a Hankel operator,  $R=H_{\varphi}$ with $\varphi\in  L^{\infty}_{s}(\p\B)$. Then, for any $f\in H^2(\p\B)$,
\begin{align*}
\mathbb P_-S R f(q) &=\mathbb{P}_-S H_\varphi f(q)=\mathbb{P}_- S  \mathbb P_-\varphi\star f(q) = \mathbb{P}_- q \star \varphi\star f (q)\\
&=\mathbb{P}_- \varphi\star q \star f (q)=H_{\varphi} S  f(q)=H_{\varphi} T  f(q)=RTf(q).
\end{align*}

Let now $R$ be an operator satisfying equation \eqref{comm}. To complete the proof we need to show that the matrix associated with $R$ with respect to the bases $\{q^n\}_{n\ge 0}$ of $H^2(\p\B)$ and $\{\bar q^n\}_{n> 0}$ of $H_-^2(\p\B)$ is a Hankel matrix.
For any $j\ge 1, k\ge 1$, 
\begin{align*}
&\left \langle Rq^j, \bar q^k\right \rangle_{ L^2_{s}(\p\B)}=\left \langle R T q^{j-1}, \bar q^k\right \rangle_{ L^2_{s}(\p\B)}=\left \langle \mathbb P_- S R q^{j-1}, \bar q^k\right \rangle_{ L^2_{s}(\p\B)}\\
&=\left \langle  S  Rq^{j-1}, \bar q^k\right \rangle_{ L^2_{s}(\p\B)}=\left \langle  R q^{j-1},  S^*\bar q^k\right \rangle_{ L^2_{s}(\p\B)}=\left \langle  R q^{j-1}, \bar q^{k+1} \right \rangle_{ L^2_{s}(\p\B)}
\end{align*}
which is the condition for $R$ to be a Hankel operator.
\end{proof}

\subsection*{Acknowledgments}
{The author wish to thank Nicola Arcozzi for useful discussions about the topic.}

\end{document}